\documentclass[a4paper]{amsart}

\usepackage{tikz}
\usepackage{tikz-cd}
\usepackage{amsfonts}
\usepackage{amsmath}
\usepackage{mathtools}
\usepackage{amsthm}
\usepackage{amssymb}
\usepackage{amscd}
\usepackage{graphicx}
\usepackage{faktor}
\usepackage[UKenglish]{babel}
\usepackage{indentfirst}
\usepackage[T2A]{fontenc}

\usepackage{color}\definecolor{darkblue}{rgb}{0,0.1,.5}
\usepackage[colorlinks=true,linkcolor=darkblue, urlcolor=darkblue, citecolor=darkblue]{hyperref}
\usepackage{cleveref}

\newcommand{\CP}{\mathbb{C}P^{\infty}}
\newcommand{\Z}{\mathcal{Z}}
\newcommand{\R}{\mathcal{R}}

\DeclareMathOperator{\Cobar}{Cobar}
\DeclareMathOperator{\Cotor}{Cotor}

\DeclareMathOperator{\lk}{lk}
\DeclareMathOperator{\st}{st}

\DeclareMathOperator{\rank}{rank}

\newcommand{\hr}[2][]{\hyperref[#2]{#1~\ref{#2}}}

\newtheorem{theorem}{Theorem}[section]
\newtheorem{proposition}[theorem]{Proposition}
\newtheorem{lemma}[theorem]{Lemma}

\theoremstyle{remark}

\theoremstyle{definition}
\newtheorem{example}[theorem]{Example}

\title{One-relator groups and algebras related to polyhedral products}

\author{Jelena Grbi{\'{c}}}
\address{School of Mathematical Sciences, University of Southampton, UK}
\email{j.grbic@soton.ac.uk}

\author{Marina Ilyasova}
\address{Department of Mathematics and Mechanics, Moscow State University, Russia}
\email{marina\_ilyasova@bk.ru}

\author{Taras Panov}
\address{Department of Mathematics and Mechanics, Moscow
State University, Russia;\newline
Institute for Information Transmission Problems, Russian Academy of Sciences, Moscow, Russia;\newline
Scientific Educational Mathematical Centre, Kazan Federal Univerisity, Russia}
\email{tpanov@mech.math.msu.su}

\author{George Simmons}
\address{School of Mathematical Sciences, University of Southampton, UK}
\email{g.j.h.simmons@soton.ac.uk}

\subjclass[2010]{Primary 20F55, 20F65, 55P15; Secondary 05E45, 57M07, 57T30}
\keywords{One-relator groups and algebras, commutator subgroups of right-angled Coxeter groups, moment-angle complexes, flag complexes}

\thanks{The third author was supported by the Russian Foundation for Basic Research grant no.~20-01-00675. The authors thank the Fields Institute for Research in Mathematical Sciences for the opportunity to work on this research project during the Thematic Program on Toric Topology and Polyhedral Products.}

\begin{document}

\maketitle

\begin{abstract}
We link distinct concepts of geometric group theory and homotopy theory through underlying combinatorics. For a flag simplicial complex~$K$, we specify a necessary and sufficient combinatorial condition for the commutator subgroup $RC_K'$ of a right-angled Coxeter group, viewed as the fundamental group of the real moment-angle complex $\mathcal{R}_K$, to be a one-relator group; 
and for the Pontryagin algebra $H_{*}(\Omega \mathcal{Z}_K)$ of the moment-angle complex to be a one-relator algebra. We also give a homological characterisation of these properties. For $RC_K'$, it is given by a condition on the homology group $H_2(\mathcal{R}_K)$, whereas for $H_{*}(\Omega \mathcal{Z}_K)$ it is stated in terms of the bigrading of the homology groups of~$\mathcal{Z}_K$.
\end{abstract}

\section{Introduction}

Let $K$ be a flag simplicial complex on vertex set $[m]=\{1,\ldots,m\}$ and $K^1$ be its $1$-skeleton. The \emph{right-angled Coxeter group} corresponding to $K$ is defined as the group $RC_K$ with generators $g_1,\ldots,g_m$ for each vertex in $K$ and relations $g_i^2=1$ and $g_ig_j = g_jg_i$ whenever $\{i,j\} \in K^1$. Right-angled Coxeter groups are interesting from a geometric point of view because they arise from reflections in the facets of right-angled polyhedra in hyperbolic space.

For a given group $G$, we denote by $G'$ the commutator subgroup of~$G$. The \emph{real moment-angle complex} $\R_K = (D^1, S^0)^K$ associated with a flag complex $K$ is a finite-dimensional aspherical space whose fundamental group is the commutator subgroup~$RC_{K}'$ of the right-angled Coxeter group $RC_{K}$.  In~\cite{pa-ve16} it was shown that $RC_K'=\pi_1 \left(\R_K \right)$ is free if and only if $K^1$ is a chordal graph. A graph is called \emph{chordal} if each of its cycles with $4$ or more vertices has a chord, an edge joining two vertices that are not adjacent in the cycle. Furthermore, for arbitrary flag $K$, a minimal generating set for $RC_K'$  was given in terms of iterated commutators of the generators of~$RC_{K}$~\cite[Theorem~4.5]{pa-ve16}.

Another space associated with a simplicial complex $K$ is the \emph{moment-angle complex} $\Z_K=(D^2,S^1)^K$. Throughout this paper, all homology groups are considered with coefficients in $\mathbb{Z}$, unless otherwise stated. The Pontryagin algebra $H_{\ast}(\Omega \Z_K)$, 
was studied in~\cite{g-p-t-w16} when $K$ is a flag complex. It was shown that $H_{\ast}(\Omega \Z_K)$ is a graded free associative algebra if and only if the 1-skeleton $K^1$ is a chordal graph. Furthermore, a minimal generating set for $H_{\ast}(\Omega \Z_K)$ with flag $K$ was given in~\cite[Theorem~4.3]{g-p-t-w16} in terms of iterated commutators. 

Therefore, for both $\R_K$ and $\Z_K$ the algebraic freeness property, that is, that $\pi_1(\R_K)$ and $H_{\ast}(\Omega \Z_K)$ are free as groups and algebras, respectively, is characterised by the same combinatorial condition. More precisely, these algebraic objects are free if and only if the 1-skeleton $K^1$ of the simplicial complex $K$ is a chordal graph. 
The question of $H_*(\Omega\Z_K)$ being a free associative algebra is related to the Golodness property of a simplicial complex $K$. A simplicial complex $K$ is \emph{Golod} if all cup products and higher Massey products vanish in $H^*(\Z_K)$. In~\cite[Theorem~4.6]{g-p-t-w16} it was proved that a flag simplicial complex $K$ is Golod if and only if $K^1$ is a chordal graph. 

In this paper we study other properties of objects naturally arising in geometric group theory and  homotopy theory that have the same combinatorial characterisation. In particular, we describe a combinatorial condition on a flag complex $K$ under which $\pi_1(\R_K)$ is a one-relator group, and $H_{\ast}(\Omega \Z_K)$ is a one-relator algebra. A 1-dimensional simplicial complex $C_p$ that is the boundary of a $p$-gon is called a $p${\it-cycle}.  In~\cite{g-p-t-w16} it was shown that when $K$ is a $5$-cycle then there is only one relation between the $10$ multiplicative generators of $H_{\ast}(\Omega \Z_K)$; while in~\cite{vere16}, a single relation was again found between the $34$ multiplicative generators of $H_{\ast}(\Omega \Z_K)$ when $K$ is a $6$-cycle. Similarly, in~\cite{pa-ve16}, it was noted that if $K$ is a $p$-cycle for $p \geqslant 4$, then $\pi_1(\R_K)$ is a one-relator group. The one-relator condition places strong restrictions on the form of $K$, and our main combinatorial characterisation is the following.

\begin{theorem}\label{mainthm} Let $K$ be a flag simplicial complex. Then
 $\pi_1(\R_K)$ and $H_{\ast}(\Omega \Z_K)$ have exactly one relation if and only if the following combinatorial condition holds
\begin{equation} \label{eq:starcond}
    K = C_p\text{ or } K = C_p \ast \Delta^q \; \text{ for } p \geqslant 4,\; q \geqslant 0 \tag{$\ast$}
\end{equation}
where $C_p$ is a $p$-cycle, $\Delta^q$ is a $q$-simplex and $*$ denotes the join of simplicial complexes.
\end{theorem}

For $\R_K$ this is proved in Theorem~\ref{thm} and for $\Z_K$ in Theorem~\ref{thm:pgoniffORA}. 
The proofs of Theorem~\ref{thm} and Theorem~\ref{thm:pgoniffORA} are completely different in character. For Theorem~\ref{thm}, the key argument comes from geometric group theory. When $K=C_p$ or $K=C_p*\Delta^q$ for $q\geqslant 0$, the space $\R_K$ is homeomorphic to the product $S_g \times D^{q+1}$, where $S_g$ is a closed orientable surface of genus $g = (p-4)2^{p-3}+1$ and $D^{q+1}$ is a $(q+1)$-dimensional disc, and therefore its fundamental group is a one-relator surface group. The converse statement is proved using the Lyndon Identity Theorem~\cite{lynd50} (see~\cite[Theorem~2.1]{dy-va73}) because the group $\pi_1(\R_K) = RC_{K}'$ is torsion-free.

To prove Theorem~\ref{thm:pgoniffORA}, we study the simply connected space $\Omega\Z_K$ using homotopy-theoretical methods. When $K=C_p$ or $K=C_p*\Delta^q$ for $q\geqslant 0$,
by a result of McGavran \cite{mcga79}, there is a homotopy equivalence 
\[
\Z_{K} \simeq \#_{k=3}^{p-1} (S^k \times S^{p+2-k})^{\# (k-2) {\genfrac(){0pt}{1}{p-2}{k-1}}}
\]
where $\#$ denotes the connected sum operation on manifolds. 
Beben and Wu~\cite{be-wu15} computed the algebra $H_{\ast}(\Omega X; \mathbb{Z}_p)$, $p$ prime, where $X$ is a highly-connected manifold obtained by attaching a single cell to a space $Y$ which has the homotopy type of a double suspension. This implies that $H^{\ast}(Y; \mathbb{Z}_p)$ has no non-trivial cup products, which places sufficient restrictions on $H^{\ast}(X; \mathbb{Z}_p)$ so that $H_{\ast}(\Omega X; \mathbb{Z}_p)$ can be studied via a homology Serre spectral sequence. 
We adapt the Beben--Wu method to study the integral Pontryagin algebra of an arbitrary connected sum of sphere products 
\begin{equation} \label{eq:ConnectedSum} 
M = \#_{i=1}^k (S^{d_i} \times S^{d-d_i})
\end{equation} 
where $d_i \geqslant 2$ and $d\geqslant 4$.
In this case, the Beben--Wu method reduces to the Adams--Hilton model and the highly-connectedness assumption can be dropped. In Proposition~\ref{prop:HopfAlgIso}, we prove that the integral Pontryagin algebra $H_{\ast}(\Omega M)$ is isomorphic as a Hopf algebra to the quotient of a graded free associative algebra by a single relation. Proposition \ref{prop:HopfAlgIso} implies that when $K=C_p$ or $K=C_p*\Delta^q$ for $q\geqslant 0$, $H_{\ast}(\Omega \Z_K)$ is a one-relator algebra. We compute the Poincar{\'e} series $P(H_{\ast}(\Omega \Z_K);t)$ explicitly in Proposition~\ref{prop:PoincareSeries}.

We extend the equivalences of Theorem~\ref{mainthm} by determining an equivalent homological criteria on $\R_K$ and $\Z_K$. For $\R_K$, the combinatorial condition~\eqref{eq:starcond} is equivalent to the homological condition $H_2(\R_K) = \mathbb{Z}$, and this is proved in Theorem~\ref{thm}. The homology groups of $\Z_K$ have a natural bigrading, see~\cite[\S~4.4]{bu-pa15}. The combinatorial condition~\eqref{eq:starcond} is then equivalent to the homological condition 
\[
H_{2-j,2j}(\Z_K) =\begin{cases} \mathbb{Z}&\text{if }j=p
\\0&\text{otherwise.}
\end{cases}
\]
This is proved in Theorem~\ref{thm:pgoniffORA}.

Although the homotopy type of a moment-angle complex $\Z_K$ is not accessible in general, various homotopy-theoretical concepts can be described if $K$ is a flag complex. Moreover, many of these homotopy-theoretical characterisations of $\Z_K$ are equivalent. For example, for $K$ a flag complex, $\Z_K$ having the homotopy type of a wedge of spheres is equivalent to $\Z_K$ being a co-H space and these concepts are equivalent to $K$ being Golod. In this paper, we show that $H_{\ast}(\Omega \Z_K)$ is a one-relator algebra if and only if $\Z_K$ has the homotopy type of a connected sum of sphere products, with two spheres in each product. Additionally, these properties are closely related to $K$ being minimally non-Golod, see~\cite{limo15}, and we summarise this relationship in Proposition~\ref{flagmng}.

We note that despite the similarity of the results for the moment-angle complex $\Z_K$ and its real analogue $\R_K$ when $K$ is flag, the techniques used in proofs differ significantly. For the case of $\Z_K$, homotopy-theoretical methods are more prevalent, whereas the case of $\R_K$ requires the use of methods in combinatorial and geometric group theory. Given a homotopy-theoretical result related to $\Z_K$, one could predict the corresponding group-theoretical result for $\R_K$, but it is an open and challenging problem to find a systematic way of translating these results directly. This is a problem of interest to both topologists and group theorists.

For non-flag $K$, all homotopy-theoretical characterisations of moment-angle complexes $\Z_K$ are more complex. For example, for an arbitrary Golod complex $K$, the moment-angle complex $\Z_K$ is not necessarily a co-H space~\cite{ir-ya17} and its cohomology can contain torsion~\cite{g-p-t-w16}. Furthermore, describing the Pontryagin algebra $H_*(\Omega\Z_K)$, and in particular determining the class of $K$ for which it is a free or one-relator algebra, is considerably harder in the non-flag case. The problem of determining those $K$ which are Golod or minimally non-Golod is also more involving, and in general distinct from studying $H_{\ast}(\Omega \Z_K)$. At the end of Section~5 we expand on the distinction between the properties of $K$ being minimally non-Golod and $H_*(\Omega\Z_K)$ being a one-relator algebra. This complexity is also seen in the real case. In the non-flag case, the real moment-angle complex $\R_K$ is not aspherical, so its topology is not determined by its fundamental group. Therefore, the question of describing $H_*(\Omega\R_K)$ does not lie entirely within combinatorial group theory.

\section{Preliminaries}
	 
Let $K$ be a \textit{simplicial complex} on the set $[m] =\{1,2,\dots,m\}$, that is, $K$ is a collection of subsets $I \subseteq [m]$ such that for any $I \in K$ all subsets of $I$ also belong to $K$. We always assume that $K$ contains $\emptyset$ and all singletons $\{i\}\in[m]$.

Let
\[
 (\textbf{X},\textbf{A}) = \{(X_1,A_1),\dots,(X_m,A_m)\}
 \]
be a sequence of $m$ pairs of pointed topological spaces, $pt \in A_i \subseteq X_i$. For each subset $I \subseteq [m]$, we set 
\[
 (\textbf{X},\textbf{A})^I = \Bigl\{ (x_1,\dots,x_m) \in \prod_{k=1}^m X_k \mid x_k \in A_k \text{ for } k \not\in I \Bigr\}
\]
and define the \textit{polyhedral product} of $(\textbf{X},\textbf{A})$ over the complex $K$ as
\[
(\textbf{X},\textbf{A})^K = \bigcup_{I \in K} (\textbf{X},\textbf{A})^I =\bigcup_{I \in K} \biggl( \prod_{i \in I} X_i \times \prod_{i \not\in I} A_i \biggr) \subseteq \prod_{k=1}^m X_k. 
\]

In the case when $X_i = X$ and $A_i = A$ for all $i$ we use the notation $(X,A)^K$ for $(\textbf{X},\textbf{A})^K$.

\begin{example}\label{rmac}\

\noindent\textbf{1.}
Let $(X,A) = (D^1, S^0)$, where $D^1$ is the closed interval $[-1,1]$ and $S^0$ is its boundary $\{-1,1\}$. The polyhedral product $(D^1, S^0)^K$ is known as the \textit{real moment-angle complex} and is denoted by $\R_K$,
$$ \R_K = (D^1, S^0)^K = \bigcup_{I \in K} (D^1, S^0)^I. $$
Note that $\R_K$ is a cubic subcomplex in the cube $(D^1)^m=[-1,1]^m$.

\smallskip

\noindent\textbf{2.}
Let $(X,A) = (D^2, S^1)$, where $D^2$ is the closed unit disc 
and $S^1$ is its boundary. The polyhedral product $(D^2,S^1)^K$ is known as the \emph{moment-angle complex} and is denoted by $\Z_K$. If $D^2$ is considered as a $CW$-complex with one cell in each dimension zero, one and two, then the moment-angle complex $\Z_K$ is a $CW$-subcomplex of the $CW$-product complex $(D^2)^m$.

\smallskip

\noindent\textbf{3.}
Let $(X,A) = (\CP,pt)$. The polyhedral product $(\CP,pt)^K$ is known as the \textit{Davis--Januszkiewicz space} and is denoted by $DJ_K$. 
\end{example}

For any subset $J \subseteq [m]$, the corresponding \textit{full subcomplex} of
$K$ is defined by
\[
  K_J = \{I \in K | I \subseteq J \}.
\]

The homology groups of the moment-angle complex $\Z_K$ have a natural bigrading arising from the bigrading in the CW-structure of~$\Z_K$, see~\cite[\S4.4]{bu-pa15},
\begin{equation}\label{bghom}
  H_k(\Z_K) \cong \bigoplus_{-i+2j=k}H_{-i,2j}(\Z_K).
\end{equation}

The bigraded components $H_{-i,2j}(\Z_K)$ can be described through the reduced simplicial homology groups of full subcomplexes~$K_J$ using Hochster's theorem, see~\cite[Theorem~4.5.8]{bu-pa15},
\begin{equation} \label{BigradedHochster}
H_{-i,2j}(\Z_K) \cong \bigoplus_{J \subseteq [m],\, |J|=j} \widetilde{H}_{j-i-1}(K_J).
\end{equation}

Similarly, the homology groups of the real moment-angle complex $\R_K$ are given by
\begin{equation}\label{homodecmop}
  H_k(\R_K) \cong \bigoplus_{J \subseteq [m]} \widetilde{H}_{k-1}(K_J)
\end{equation}
for any $k \geqslant 0$,  see~\cite[\S4.5]{bu-pa15}.

A \textit{missing face} of $K$ is a subset $I \subseteq [m]$ such that $I$ is not a simplex of $K$, but every proper subset of $I$ is a simplex of $K$. A simplicial complex $K$ is called a \textit{flag complex} if each of its missing faces consists of two vertices, that is, any set of vertices of $K$ which are pairwise connected by edges spans a simplex. A \textit{clique} of a graph $\Gamma$ is a subset $I$ of vertices pairwise connected by edges. For a graph $\Gamma$, we define the \textit{clique complex} of $\Gamma$ as the simplicial complex obtained by filling in each clique of~$\Gamma$ by a simplex. Each flag complex $K$ is the clique complex of its $1$-skeleton $\Gamma = K^1$.

A graph $\Gamma$ is called \textit{chordal} if each of its cycles with 4 or more vertices has a chord, an edge joining two vertices that are not adjacent in the cycle. A \textit{$p$-cycle} is the same as the boundary of a $p$-gon. It is a chordal graph only when $p=3$. The simplicial complex which is a $p$-cycle is denoted by $C_p$.

If $K = C_p$, $p\geqslant4$, by a result of McGavran \cite{mcga79}, there is a homeomorphism
\begin{equation} \label{eq:McGavran} 
\Z_{K} \cong \#_{k=3}^{p-1} (S^k \times S^{p+2-k})^{\# (k-2) {\genfrac(){0pt}{1}{p-2}{k-1}}}.
\end{equation}
The corresponding real moment-angle complex $\R_K$ is an orientable surface of genus $1+(p-4)2^{p-3}$, see~\cite[Proposition~4.1.8]{bu-pa15}.

The algebra $H_{\ast}(\Omega \Z_K)$ was studied in \cite{pa-ra04,g-p-t-w16}. The homotopy fibration $\Z_K\to DJ_K\to (\mathbb C P^\infty)^m$ gives rise to a short exact sequence of Hopf algebras
\begin{equation}\label{haes} \begin{tikzcd}
  1 \ar[r] & H_*(\Omega\Z_K) \ar[r] & H_*(\Omega DJ_K) \ar[r,"\mathrm{Ab}"] & 
  \Lambda[u_1, \ldots,u_m] \ar[r] & 0
\end{tikzcd} \end{equation} 
where $\mathrm{Ab}$ is the ``abelianisation'' homomorphism to the graded commutative algebra $\Lambda[u_1,\ldots,u_m]=H_*(\Omega(\mathbb C P^\infty)^m)$ with $\deg u_i=1$. The algebra 
$H_*(\Omega\Z_K)$ can be viewed as the commutator subalgebra of $H_{\ast}(\Omega DJ_K)$. Let $[a,b] = ab + (-1)^{\deg a \deg b +1} ba$ denote the graded Lie commutator of the elements $a$ and $b$. In the case that $K$ is flag, there is an algebra isomorphism~\cite[Theorem~9.3]{pa-ra04}
\begin{equation} \label{DJ_Kalg}
  H_{\ast}(\Omega DJ_K) \cong T(u_1,\dots,u_m) / \langle u_i^2,\; [u_i,u_j] \text{ if } \{i,j\} 
  \in K  \rangle
\end{equation}
where $T(u_1,\dots,u_m)$ is a graded free associative algebra and $\deg u_i = 1$. A minimal multiplicative generating set for $H_{\ast}(\Omega \Z_K)$ is given as in \cite[Theorem~4.3]{g-p-t-w16}. Namely, $H_{\ast}(\Omega \Z_K)$ is multiplicatively generated by $\sum_{J \subseteq [m]}\mathop{\mathrm{rank}}\widetilde{H}_{0}(K_J)$ iterated commutators of the form
\begin{equation} \label{MACgenerators}
 [u_j,u_i],~ [u_{k_1},[u_j,u_i]],~ \dots,~ [u_{k_1},[u_{k_2},\dots,[u_{k_{l-2}},[u_j,u_i]]\dots]]
\end{equation}
where $k_1<k_2<\dots<k_{l-2}<j>i$, $k_s\not=i$ for any $s$ and $i$ is the smallest vertex in a connected component of $K_{\{k_1,\dots,k_{l-2},j,i\}}$ not containing~$j$. Additionally, it was shown in \cite[Theorem~4.6]{g-p-t-w16} that $H_{\ast}(\Omega \Z_K)$ is a free associative algebra if and only if the graph $K^1$ is chordal, in which case $\Z_K$ is homotopy equivalent to a wedge of spheres.

Parallel results for the real moment-angle complex $\R_K$ were obtained in~\cite{pa-ve16} in the group-theoretical setting. Let $(g,h)=g^{-1}h^{-1}gh$ denote the group commutator of 
elements $g$ and $h$.

The \textit{right-angled Coxeter group} $RC_K$ corresponding to $K$ is defined by
\[ 
RC_K = F(g_1,\dots,g_m) / (g_i^2,~ (g_i,g_j) \text{ if } \{i,j\} \in K) 
\]
where $F(g_1,\dots,g_m)$ is a free group with $m$ generators.
Note that  $RC_{K}$ depends only on the $1$-skeleton $\Gamma=K^1$, which is a graph.

Recall that a path-connected space $X$ is \textit{aspherical} if $\pi_i(X)=0$ for $i \geqslant 2$. An aspherical space $X$ is an Eilenberg--Mac Lane space $K(\pi,1)$ with $\pi = \pi_1(X)$. The following facts relating the real moment-angle complex $\R_K$ to the right-angle Coxeter group~$RC_K$ are known, see, for example~\cite[Corollary~3.4]{pa-ve16}):
\begin{itemize}
	\item[(i)] $\pi_1(\R_K)$ is isomorphic to the commutator subgroup $RC'_{K}$;
	\item[(ii)] $\R_K$ is aspherical if and only if $K$ is flag.
\end{itemize}
Therefore, in the flag case the algebra $H_*(\Omega\R_K)$ reduces to the non-abelian group $H_0(\Omega\R_K)=RC'_K$ and the analogue of~\eqref{haes} is the short exact sequence of groups
\begin{equation*} \begin{tikzcd}
  1 \ar[r] & RC'_K \ar[r] & RC_K \ar[r,"\mathrm{Ab}"] & (\mathbb{Z}_2)^m \ar[r] & 0
\end{tikzcd} \end{equation*}
where $\mathbb{Z}_2$ is an elementary abelian $2$-group and $\mathrm{Ab}$ is the abelianisation homomorphism.

By analogy with~\cite{g-p-t-w16}, the following combinatorial criterion 
was obtained in~\cite[Corollary~4.4]{pa-ve16}: 
the commutator subgroup $RC'_{K}$ is a free group if and only if the graph $K^1$ is chordal.

An explicit minimal generator set for the commutator subgroup $RC'_{K}$ is described in~\cite[Theorem~4.5]{pa-ve16}. 
It consists of $\sum_{J \subseteq [m]}\mathop{\mathrm{rank}}\widetilde{H}_{0}(K_J)$ nested commutators
\begin{equation}\label{generators}
 (g_j,g_i),~~ (g_{k_1},(g_j,g_i)),~ \dots,~ (g_{k_1},(g_{k_2},\dots,(g_{k_{l-2}},(g_j,g_i))\dots))
\end{equation}
where $k_1<k_2<\dots<k_{l-2}<j>i$, $k_s\not=i$ for any $s$, and $i$ is the smallest vertex in a connected component of $K_{\{k_1,\dots,k_{l-2},j,i\}}$ not containing~$j$.

\section{One-Relator Groups}

A group $G$ is called \textit{a one-relator group} if $G$ is not a free group and can be presented with a generating set with a single relation.

Let $G$ be a one-relator group, that is, $G=F/R$, where $F=F(x_1,\dots,x_l)$ is a free group and $R$ is the smallest normal subgroup in~$F$ generated by relation~$r$. Consider the space
\begin{equation}\label{yg}
Y(G) = \Bigl(\bigvee_{i=1}^lS_i^1\Bigr) \cup_{\bar r} e^2
\end{equation}
obtained by attaching a $2$-cell  to a wedge of circles via a map $\bar r\colon S^1\rightarrow\bigvee S_i^1$ corresponding to the element~$r\in F$.

Recall that all homology groups are considered with coefficients in $\mathbb{Z}$. The homology groups of $Y(G)$ are described as follows.

\begin{proposition}\label{prop}
	$H_k(Y(G))=0$ for $k\geqslant 3$, $H_1(Y(G))=\mathbb Z^l$ and
	\[
	H_2(Y(G))=
	\begin{cases}
	\mathbb{Z} & \text{if $r \in [F, F]$} \\
	0 & \text{otherwise.}
	\end{cases}
	\] \qed
\end{proposition}

Lyndon~\cite{lynd50} studied cohomology theory of groups with a single relation by considering the corresponding space $Y(G)$. Dyer and Vasquez~\cite{dy-va73} gave an equivalent formulation of the Lyndon Identity Theorem in the following form (see~\cite[Theorem~2.1]{dy-va73}]): if $G$ is a one-relator group with relation~$r$ which is not a proper power, that is,  $r \ne u^n$ for $n>1$, then $Y(G)$ is a $K(G,1)$-space.

Under the conditions of the Lyndon Identity Theorem, we have $H_k(G;\mathbb{Z})=H_k(Y(G);\mathbb{Z})$, that is, the homological dimension of $G$ is at most~2.

\begin{theorem}\label{thm}
Let $K$ be a flag simplicial complex on $[m]$. The following conditions are equivalent:
\begin{itemize}
\item[(a)] $\pi_1(\R_K) = RC'_{K}$ is a one-relator group;	
\item[(b)] $H_2(\R_K) = \mathbb{Z}$;
\item[(c)] $K = C_p$ or $K = C_p \ast \Delta^q$ for $p \geqslant 4$ and $q \geqslant 0$, where $C_p$ is a $p$-cycle, $\Delta^q$ is a $q$-simplex, and $*$ denotes the join of simplicial complexes.
\end{itemize}
If any one of these conditions is met, we have $H_k(\R_K)=0$ for $k\geqslant 3$.
\end{theorem}

\begin{proof}
(c) $\Rightarrow$ (b). This implication follows from the implications below, but we include an independent proof as it is simple and illustrative. Suppose that $K=C_p$ or $K=C_p*\Delta^q$ for $p\geqslant 4$ and $q\geqslant 0$.  Let the $p$-cycle $C_p$ be supported on the set of vertices $I = \{i_1,\dots,i_p\}$. By homology decomposition~\eqref{homodecmop},
\begin{equation}\label{h2}
H_2(\R_K) \cong \bigoplus_{J \subseteq [m]} \widetilde{H}_{1}(K_J).
\end{equation}
Since $K_I$ is a $p$-cycle, we have $\widetilde{H}_{1}(K_I) = \mathbb{Z}$. Because any subcomplex $K_J$ with $J \not= I$ is contractible, $\widetilde{H}_{1}(K_J) = 0$ for $J \not= I$.
It follows that $H_2(\R_K) = \mathbb{Z}$.

\smallskip

(b) $\Rightarrow$ (c). Suppose $H_2(\R_K) = \mathbb{Z}$. Then only one summand in~\eqref{h2} is $\mathbb{Z}$, and all other summands are zero. Since $K$ is a flag complex, this implies that there exists a set of vertices $I = \{i_1,\dots,i_p\}$ such that $K_I$ is a $p$-cycle with $p \geqslant 4$. Since $\widetilde{H}_{1}(K_J) = 0$ for any proper subset $J \subseteq I$, any two vertices which are not adjacent in the $p$-cycle are not connected by an edge. If there exists a vertex $j \not\in I$ in the complex $K$, then $\widetilde{H}_{1}(K_{I \cup \{j\}}) = 0$ implies that the vertex $j$ is connected to each vertex in the $p$-cycle~$I$. If $K$ has two vertices $j_1, j_2 \not\in I$ which are not connected by an edge, then the subcomplex $K_{\{i_1, i_3\} \cup \{j_1, j_2\}}$ is a $4$-cycle and $\widetilde{H}_{1}(K_{\{i_1, i_3\} \cup \{j_1, j_2\}}) = \mathbb{Z}$, which contradicts the assumption. Hence, all vertices of $K$ which are not in the set $I$ are connected to each other and to all vertices of~$I$. Since $K$ is a flag complex, we obtain $K =C_p*\Delta^q$ for some $p \geqslant 4$ and $q \geqslant 0$.

\smallskip

(c) $\Rightarrow$ (a). First let $K$ be a $p$-cycle $C_p$. In this case the complex $\R_K$ is homeomorphic to a closed orientable surface of genus $(p-4)2^{p-3}+1$ (see~\cite[Proposition~4.1.8]{bu-pa15}). Also, $\pi_1(\R_K) \cong RC'_{K}$. Hence, $RC'_{K}$ is a one-relator group. 

Now let $\widetilde{K} = K*\Delta^q$, where $K$ is a $p$-cycle. Then $\R_{\widetilde{K}} = \R_K \times D^{q+1}$ and $RC'_{\widetilde{K}} =\pi_1(\R_{\widetilde{K}})=\pi_1(\R_{K})=RC'_{K}$ is a one-relator group.

\smallskip

(a) $\Rightarrow$ (b). Since $\R_K$ is an aspherical finite cell complex, the group $\pi_1(\R_K)$ is torsion-free (for example, see~\cite[Proposition~2.45]{hatc11}). So if $\pi_1(\R_K)=F/R$ is a one-relator group with a relation $r$, then $r$ is not a proper power $u^n$ for $n>1$, as otherwise the element $u$ would be of finite order. 

Consider the space $Y(RC'_{K})$, constructed as in~\eqref{yg}. According to the Lyndon Identity Theorem,  $Y(RC'_{K})$ is homotopy equivalent to $K(RC'_{K},1)$, so its homology groups coincide with the homology groups of the space~$\R_K$.
Proposition~\ref{prop} implies that $H_2(\R_K)$ is either $\mathbb{Z}$ or~$0$. The group $RC'_{K}$ is not free, so the graph $K^1$ is not chordal, that is, there exists a chordless cycle on $I$ of length $p\geqslant 4$. Therefore, one of the summands on the right hand side of~\eqref{h2}  is equal to $\mathbb{Z}=\widetilde{H}_{1}(K_I)$.
Thus, $H_2(\R_K) = \mathbb{Z}$.

\smallskip

It remains to prove that (c) implies that $H_k(\R_K)=0$ for $k\geqslant 3$. Considering homology decomposition~\eqref{homodecmop}, 
\[
  H_k(\R_K) \cong \bigoplus_{J \subseteq [m]} \widetilde{H}_{k-1}(K_J)
\]
we claim that all summands with $k\geqslant 3$ on the right hand side are equal to~0. Indeed, let $I = \{i_1,\dots,i_p\}$ be the set of vertices of $K$ forming a $p$-cycle. Then 
$\widetilde{H}_{k-1}(K_I) = 0$ for $k\geqslant 3$. Since any full subcomplex 
$K_J$ with $J\ne I$ is contractible, we get $\widetilde{H}_{k-1}(K_J) = 0$. Hence, $H_k(\R_K)=0$ for $k\geqslant 3$.
\end{proof}


The following examples illustrate Theorem~\ref{thm}.

\begin{example}\label{examples}\
	
\noindent\textbf{1.} Let $K$ be the flag complex in Figure~\ref{fig}~(a).
\begin{figure}[t]
    \centering
    \begin{tikzpicture}
        \coordinate [label=left:$1$] (1) at (0,2);
        \coordinate [label=right:$2$] (2) at (2,2);
        \coordinate [label=right:$3$] (3) at (2,0);
        \coordinate [label=left:$4$] (4) at (0,0);
        \coordinate [label=below:$5$] (5) at (1,1);
         \coordinate [label = below: $(a)$] (a) at (1,-0.5);
        \filldraw[fill=black!20] (2) -- (5) -- (1);
        \filldraw[fill=black!20] (2) -- (5) -- (3);
        \draw (2) -- (1) -- (4) -- (3) -- (2);
        \foreach \point in {1,2,3,4,5}
            \fill [black] (\point) circle (1.5 pt);
        \end{tikzpicture}
\hspace{0.1\textwidth}
    \begin{tikzpicture}
        \coordinate [label=left:$1$] (1) at (0,2);
        \coordinate [label=right:$2$] (2) at (2,2);
        \coordinate [label=right:$3$] (3) at (2,0);
        \coordinate [label=left:$4$] (4) at (0,0);
        \coordinate [label = below: $5$] (5) at (1,1);
        \coordinate [label = below: $(b)$] (a) at (1,-0.5);
        \filldraw[fill=black!20] (2) -- (5) -- (1);
        \filldraw[fill=black!20] (2) -- (5) -- (3);
        \filldraw[fill=black!20] (1) -- (5) -- (4);
        \filldraw[fill=black!20] (4) -- (5) -- (3);
        \draw (2) -- (1) -- (4) -- (3) -- (2);
        \foreach \point in {1,2,3,4,5}
            \fill [black] (\point) circle (1.5 pt);
        \coordinate [label = below: $5$] (5) at (1,1);
        \end{tikzpicture}
   \caption{}
    \label{fig}
\end{figure}
Generator set~\eqref{generators} for the commutator subgroup
$RC'_{K}$ is 
	$$(g_3,g_1), ~(g_4,g_2), ~(g_5,g_4), ~(g_2,(g_5,g_4)).$$
These satisfy the relations
\[
  (g_3,g_1)^{-1}(g_4,g_2)^{-1}(g_3,g_1)(g_4,g_2)=1, \quad 
  (g_3,g_1)^{-1}(g_5,g_4)^{-1}(g_3,g_1)(g_5,g_4)=1
\]
and
\[
  (g_3,g_1)^{-1}(g_2,(g_5,g_4))^{-1}(g_3,g_1)(g_2,(g_5,g_4)) = 1.
\]
Indeed, since each of $g_1$ and $g_3$ commutes with each of $g_2$ and $g_4$, the commutators $(g_4,g_2)^{-1}$ and $(g_3,g_1)$ commute too. We therefore obtain 
$$
  (g_3,g_1)^{-1}(g_4,g_2)^{-1}(g_3,g_1)(g_4,g_2)=(g_3,g_1)^{-1}(g_3,g_1)(g_4,g_2)^{-1}  
  (g_4,g_2)=1.
$$ 
The other two relations are proved similarly. Using homology decomposition~\eqref{homodecmop}, we get $H_2(\R_K) = \mathbb{Z}^3$. 
	
	\smallskip
	
\noindent\textbf{2.} Let $K$ be the flag complex in Figure~\ref{fig}~(b). Generator set~\eqref{generators} for $RC'_{K}$ is
\[
	(g_3,g_1), ~(g_4,g_2),
\]
which satisfy a single relation $(g_3,g_1)^{-1}(g_4,g_2)^{-1}(g_3,g_1)(g_4,g_2)=1$. Here $RC'_{K}$ is a one-relator group and $H_2(\R_K) = \mathbb{Z}$.
\end{example}

\section{Connected Sums of Sphere Products}

Let $M = \#_{i=1}^k \left( S^{d_i} \times S^{d-d_i} \right)$, where $d_i\geqslant2$, $d\geqslant4$ and $\#$ denotes the connected sum operation on manifolds. Topologically, such connected sums are obtained by attaching a single cell to a wedge of spheres, that is, there is a cofibration sequence
\begin{equation} \label{eq:CofibSeq} \begin{tikzcd}[column sep=3em,row sep=3em]
S^{d-1} \ar[r,"w"] & \bigvee_{i=1}^k S^{d_i} \vee S^{d-d_i} \ar[r,"i"] & \#_{i=1}^k (S^{d_i} \times S^{d-d_i})
\end{tikzcd} \end{equation}
where $w$ is the sum of Whitehead products $w_i \colon S^{d-1} \to S^{d_i} \vee S^{d-d_i}$. Denote by $\overline{M}$ the wedge $\bigvee_{i=1}^d S^{d_i} \vee S^{d-d_i}$. Then by the Bott--Samelson theorem $H_{\ast}(\Omega \overline{M}) \cong T(a_1,b_1,\ldots,a_k,b_k)$, where $\deg(a_i)=d_i-1$ and $\deg(b_i) = d-d_i-1$. The looped inclusion $\Omega i \colon \Omega \overline{M} \to \Omega M$ induces a map of algebras 
\[
(\Omega i)_{\ast} \colon T(a_1,b_1,\ldots,a_k,b_k) \longrightarrow H_{\ast}(\Omega M).
\]
The adjoint $\overline{w} \colon S^{d-2} \to \Omega \left( \bigvee_{i=1}^k S^{d_i} \vee S^{d-d_i} \right)$ of the sum of Whitehead products $w$ induces a map $\overline{w}_{\ast} \colon H_{d-2}(S^{d-2}) \to H_{d-2}(\Omega \overline{M})$ which sends the canonical generator to the element $\chi = [a_1,b_1] + \cdots + [a_k,b_k]$. In particular, $\chi$ is primitive and $(\Omega i)_{\ast}(\chi) = 0$ in $H_{\ast}(\Omega M)$. Then the algebra
\begin{equation} \label{eq:DefAlgT} 
\cfrac{T(a_1,b_1, \ldots, a_k,b_k)}{\langle [a_1,b_1] + \cdots + [a_k,b_k] \rangle} 
\end{equation} 
is a primitively generated Hopf algebra, where the quotient ideal is two-sided, and the algebra map $(\Omega i)_{\ast}$ factors as a map of Hopf algebras
\begin{equation} \label{eq:HopfAlgInd}
 \begin{tikzcd}[column sep=3em,row sep=3em]
T(a_1,b_1,\ldots,a_k,b_k) \ar[r,"(\Omega i)_{\ast}"] \ar[d] & H_{\ast}(\Omega M) \\
\cfrac{T(a_1,b_1, \ldots, a_k,b_k)}{\langle [a_1,b_1] + \cdots + [a_k,b_k] \rangle}  \ar[ur,"\theta"']
\end{tikzcd} \end{equation}
defining the map $\theta$.

The loop homology Hopf algebra $H_*(\Omega X;\mathbb Z)$ of a simply connected CW-complex 
$X$ can be calculated as homology of the \emph{cobar construction} $\Cobar C_*(X)$ of the reduced singular chains $C_*(X)$~\cite{adam56}, or as homology of the \emph{Adams--Hilton model}~\cite{ad-hi56} based on cells and attaching maps.

The cobar construction $\Cobar$ is a functor
\[
  \Cobar\colon\text{\sc{dgc}}_1\longrightarrow\text{\sc{dga}}
\]
from the category $\text{\sc{dgc}}_1$ of simply connected differential graded (dg) coalgebras to dg algebras. It assigns to a dg coalgebra $(C,\partial)$ with $C_0=\mathbb Z$ and $C_1=0$ the dg algebra 
\[
  \Cobar C=(F(C), d)
\]
where $F(C)=T(s^{-1}\overline C)$ is the free associative algebra on the desuspended module $\overline C=C/\mathbb Z$, the cokernel of the coaugmentation $\mathbb Z\to C$. The differential $d$ is given~by 
\begin{equation}\label{dcobar}
  d c = -\partial c+\sum_{i=2}^{p-2}(-1)^i\Delta_{i,p-i}c
\end{equation}
where $c\in s^{-1} \overline C_p$ with comultiplication  $\Delta c=c\otimes 1+1\otimes c+\sum_{i=2}^{p-2}\Delta_{i,p-i}c$.

Adams~\cite{adam56} proved that for a simply connected CW-complex $X$ there is an isomorphism of Hopf algebras
\[
  H_*(\Omega X)\cong H(\Cobar C_*(X),d)=\Cotor\nolimits_{C_*(X)}(\mathbb Z,\mathbb Z)
\]
where $C_*(X)$ is the reduced singular chain coalgebra of~$X$. 

The Adams--Hilton model ~\cite{ad-hi56} is a smaller dg algebra $AH_*(X)$ quasi-isomorphic to~$\Cobar C_*(X)$; it has generators corresponding to the cells of~$X$ and differential defined via the attaching maps.

The following statement generalises~\cite[Corollary~2.4]{ad-hi56}.

\begin{proposition} \label{prop:HopfAlgIso}
For $d_i \geqslant 2$ and $d \geqslant 4$, there is an isomorphism of Hopf algebras 
\[
H_{\ast} \left( \Omega \left( \#_{i=1}^k S^{d_i} \times S^{d-d_i} \right) \right) \cong \cfrac{T(a_1,b_1, \ldots, a_k,b_k)}{\langle [a_1,b_1] + \cdots + [a_k,b_k] \rangle}
\]
where $\deg a_i = d_i - 1$, $\deg b_i = d - d_i - 1$, and $[a_i,b_i]=a_i\otimes b_i+(-1)^{\deg a_i\deg b_i+1}b_i\otimes a_i$ is the graded commutator.
\end{proposition}

\begin{proof}
We consider the Adams--Hilton model of $M=\#_{i=1}^k S^{d_i} \times S^{d-d_i}$. 
The cofibration sequence \eqref{eq:CofibSeq} gives a CW-structure on $M$ consisting of cells $e^0,e^{d_i}_i,e^{d-d_i}_i$, $1 \leqslant i \leqslant k$, each attached trivially, and a single cell $e^d$ attached by the sum of Whitehead products $w_i \colon S^{d-1} \to S^{d_i} \vee S^{d-d_i}$. 
The Adams--Hilton model $AH_*(M)$ can be identified with the cobar construction on the coalgebra generated by positive-dimensional cells, in which the differential is zero, $e_i^{d_i}$ and $e_i^{d-d_i}$ are primitives, and
\begin{equation}\label{deltaed}
  \Delta e^d=e^d\otimes 1+1\otimes e^d+\sum_{i=1}^k
  \bigl( e^{d_i}_i\otimes e^{d-d_i}_i+(-1)^{d_i(d-d_i)}e^{d-d_i}_i\otimes e^{d_i}_i\bigr).
\end{equation}
The Adams--Hilton model is therefore
\[
  AH_*(M)=\bigl(T(a_1,b_1,\ldots,a_k,b_k,z),d\bigr)
\]
where $a_i=(-1)^{d_i}s^{-1}e^{d_i}_i$, $b_i=s^{-1}e^{d-d_i}_i$, $z=s^{-1}e^{d}$ 
and $\deg a_i = d_i - 1$, $\deg b_i = d - d_i-1$ and $\deg z = d-1$. Differential~\eqref{dcobar} is given by $d(a_i) = d(b_i) = 0$ and
\begin{align*}
  d(z) &= \sum_{i=1}^k\bigl( (-1)^{d_i} s^{-1}e^{d_i}_i\otimes s^{-1}e^{d-d_i}_i+
  (-1)^{d-d_i}(-1)^{d_i(d-d_i)} s^{-1}e^{d-d_i}_i\otimes s^{-1}e^{d_i}_i\bigr)\\
  &=\sum_{i=1}^k\bigl( a_i\otimes b_i+
  (-1)^{(d_i+1)(d-d_i)+d_i}b_i\otimes a_i\bigr)=\sum_{i=1}^k[a_i,b_i].  
\end{align*}
A nonzero $x \in AH_*(M)$ is a cycle if and only if $x$ is not in the two-sided ideal $\langle z \rangle$, and $x$ is a boundary if and only if $x \in \langle d(z) \rangle$. Therefore, homology of $\Omega M$ is as stated.
\end{proof}

For a graded vector space $V$, denote by $P(V;t)$ the Poincar{\'e} series of $V$.
 
\begin{proposition} \label{prop:PoincareSeries}
There is the following identity for the Poincar{\'e} series
\[
P \left( H_{\ast} \left( \Omega \left( \#_{i=1}^k S^{d_i} \times S^{d-d_i} \right) ; t \right) \right) = \frac{1}{1 - \sum_{i=1}^k (t^{d_i-1} + t^{d-d_i-1}) + t^{d-2}}.
\]
\end{proposition}

\begin{proof}
Let $A = H_{\ast} \left( \Omega \left( \#_{i=1}^k S^{d_i} \times S^{d-d_i} \right) \right)$. By Proposition \ref{prop:HopfAlgIso}, $A$ is the quotient of the free associative algebra on the graded set $S=\{a_1,b_1,\dots,a_k,b_k\}$, where $\deg a_i = d_i-1$ and $\deg b_i = d-d_i-1$, by the two-sided ideal generated by the element
\[
\chi = \sum_{i=1}^k [a_i,b_i] = a_1 b_1 + (-1)^{\deg a_1 \deg b_1 + 1}b_1a_1 + \sum_{i=2}^k [a_i,b_i].
\]

Let $B$ be the graded free monoid on $S$. Then $B_n$, the $n$th graded component of $B$, is a generating set for $A_n$, the $n$th graded component of $A$. For any monomial $x \in A \setminus \{1\}$, write $x = sy$ for some unique $s \in S$ and $y \in B_{n-\deg s}$. If $x = a_1b_1y'$ then using relation $\chi$ we rewrite
\[
x = \Bigl( (-1)^{\deg a_1 \deg b_1} b_1a_1 - \sum_{i=2}^k [a_i,b_i] \Bigr) y'.
\]

Let $B'_n$ be the set of all elements in $B_n$ which do not start with $a_1b_1$. By induction, $B_n'$ is a minimal generating set for $A_n$. Define $c_n = |B_n'| = \rank A_n$ for $n \geqslant 1$, $c_n = 0$ for $n < 0$, and $c_0 = 1$. From the above description, $c_n$ satisfies the recurrence formula 
\[
c_{n} = \sum_{i=1}^k (c_{n-d_i+1} + c_{n-d+d_i+1}) - c_{n-d+2}
\]
for $n \geqslant 1$. Multiplying by $t^n$ and summing over $n>0$ gives
\begin{align*}
    P(A;t) - 1 & = \sum_{n=1}^{\infty} c_n t^n \\
    & = \sum_{n=1}^{\infty} \left( \sum_{i=1}^k (c_{n-d_i+1} + c_{n-d+d_i+1}) - c_{n-d+2} \right) t^n \\
    & = \sum_{i=1}^k \sum_{n=2-d_i}^{\infty} c_n t^{n+d_i-1} + \sum_{i=1}^k \sum_{n=2-d+d_i}^{\infty} c_n t^{n+d-d_i-1} - \sum_{n=3-d}^{\infty} c_n t^{n+d-2} \\
    & = \left(\sum_{i=1}^k (t^{d_i-1} + t^{d-d_i-1}) - t^{d-2} \right) \sum_{n=0}^{\infty} c_n t^n \\
    & = \left(\sum_{i=1}^k (t^{d_i-1} + t^{d-d_i-1}) - t^{d-2} \right) P(A; t)
\end{align*}
which is rearranged to give the claimed identity.
\end{proof}

\section{One-Relator Algebras}

An algebra is a \textit{one-relator algebra} if it is not free and can be written as the quotient of a free associative algebra by a two-sided ideal generated by a single element.

We recall the bigraded decomposition~\eqref{BigradedHochster} of the integral homology of the moment-angle complex~$\Z_K$.

\begin{theorem} \label{thm:pgoniffORA}
Let $K$ be a flag simplicial complex on~$[m]$. The following conditions are equivalent:
\begin{itemize}
\item[(a)] $H_{\ast}(\Omega \Z_K)$ is a one-relator algebra;	
\item[(b)] there exists $p$ with $4\leqslant p\leqslant m$ such that $H_{2-j,2j}(\Z_K) =\begin{cases} \mathbb{Z}&\text{if }j=p;
\\0&\text{otherwise;}
\end{cases}$
\item[(c)] $K = C_p$ or $K = C_p \ast \Delta^q$ for $p \geqslant 4$ and $q \geqslant 0$, where $C_p$ is a $p$-cycle and  $\Delta^q$ is a $q$-simplex.
\end{itemize}
If any one of these conditions is met, we have $H_{-i,2j}(\Z_K)=0$ for $j-i\geqslant 3$.
\end{theorem}

To prove the Theorem, we start by showing that if $K$ is a flag complex which is not of the form given in~(c), then either $H_{\ast}(\Omega \Z_K)$ is free, or it has at least two relations. The following result gives a condition for $H_{\ast}(\Omega \Z_K)$ to have at least two relations.

\begin{lemma} \label{lm:DistinctRetract}
Let $K$ be a simplicial complex and suppose that $K_I$ and $K_J$ are distinct full subcomplexes of $K$ such that both $H_{\ast}(\Omega \Z_{K_I})$ and $H_{\ast}(\Omega \Z_{K_J})$ have at least one relation. Then $H_{\ast}(\Omega \Z_K)$ is not a one-relator algebra.
\end{lemma}

\begin{proof}
Note that each of $\Z_{K_I}$ and $\Z_{K_J}$ retracts off $\Z_K$ as $K_I$ and $K_J$ are full subcomplexes. 
Therefore each of $\Omega \Z_{K_I}$ and $\Omega \Z_{K_J}$ retracts off $\Omega \Z_K$ and we obtain a commutative diagram of algebras
\begin{equation*} \begin{tikzcd}[column sep=3em,row sep=3em]
H_{\ast}(\Omega \Z_{K_I}) \ar[r] \ar[dr,equal] & H_{\ast}(\Omega \Z_K) \ar[d]\\
& H_{\ast}(\Omega \Z_{K_I})
\end{tikzcd} \end{equation*}
and similarly for $K_J$. In particular, each relation of $H_{\ast}(\Omega \Z_{K_I})$ appears as a relation of $H_{\ast}(\Omega \Z_K)$ under the induced inclusion map and similarly for $K_J$, and the induced relations are distinct since $K_I$ and $K_J$ are.
\end{proof}

Let $K$ be a simplicial complex on $[m]$. Suppose that $j \in K$ is a vertex and define the link 
\[
  \lk_K(j) = \{I \in K \mid j \cup I \in K,\, j \notin I \}
\]
and the star
\[
\st_K(j) = \{I \in K \mid j \cup I \in K \} = \lk_K(j) \ast j
\]
 and assume that $\lk_K(j)$ is on the first $l$ vertices of $K$. Decompose $K=\st_K(j) \cup_{\lk_K(j)} K_{[m] \setminus j}$. Then there is a homotopy pushout of moment-angle complexes
\begin{equation} \label{linkstar}
\begin{tikzcd}[column sep=3em,row sep=3em]
  \Z_{\lk_K(j)} \times T^{m-l} \ar[r,"i \times \pi"] \ar[d,"\mathrm{id} \times \pi"'] 
  & \Z_{K_{[m] \setminus j}} \times S^1 \ar[d] \\
  \Z_{\lk_K(j)}   \times T^{m-l-1} \ar[r] & \Z_{K}.
\end{tikzcd} 
\end{equation}

\begin{lemma}\label{gtlemma}
If the map $i \colon  \Z_{\lk_K(j)} \to \Z_{K_{[m] \setminus j}}$ is nullhomotopic, then there is a homotopy equivalence 
\[
\Z_K \simeq \Sigma^2 (\Z_{\lk_K(j)} \times T^{m-l-1}) \vee (\Z_{[m] \setminus j} \rtimes S^1).
\]
Here the half-smash $X\rtimes Y$ of pointed spaces is defined by $X\times Y/(\mathit{pt}\times Y)$.
\end{lemma}
\begin{proof}
This is a particular case of~\cite[Lemma~3.3]{gr-th07}.
\end{proof}

The following result shows that when $\Z_{[m] \setminus j}$ has the homotopy type of a connected sum of sphere products, $H_{\ast}(\Omega \Z_K)$ is not a one-relator algebra.

\begin{lemma} \label{lm:HalfSmashGivesRelator}
Suppose that $M = \#_{i=1}^k \left( S^{d_i} \times S^{d-d_i} \right)$ where $d_i \geqslant 2$ and $d \geqslant 4$. Then $H_{\ast}(\Omega (M \rtimes S^1))$ is not a one-relator algebra.
\end{lemma}

\begin{proof}
As in Proposition~\ref{prop:HopfAlgIso}, we apply the Adams--Hilton model.
A cell structure on $M \rtimes S^1$ is given by the image under the quotient map $M \times S^1 \to M \rtimes S^1$, and therefore consists of cells $e^0$, $e_i^{d_i},e_i^{d-d_i},e_i^{d_i+1},e_i^{d-d_i+1}$, $1 \leqslant i \leqslant k$, along with two cells $e^d$ and $e^{d+1}$. The Adams--Hilton model $AH_*(M \rtimes S^1)$ can be identified with the cobar construction on the coalgebra generated by positive-dimensional cells, in which the differential is zero, $e_i^{d_i},e_i^{d-d_i},e_i^{d_i+1},e_i^{d-d_i+1}$ are primitives, $\Delta e^d$ is given by~\eqref{deltaed} and
\begin{multline*}
  \Delta e^{d+1}=e^{d+1}\otimes 1+1\otimes e^{d+1}+\sum_{i=1}^k
  \bigl( e^{d_i}_i\otimes e^{d-d_i+1}_i+(-1)^{d_i(d-d_i+1)}e^{d-d_i+1}_i\otimes e^{d_i}_i\\
  +(-1)^{d-d_i}e^{d_i+1}_i\otimes e^{d-d_i}_i+(-1)^{d_i(d-d_i)}e^{d-d_i}_i\otimes e^{d_i+1}_i
  \bigr).
\end{multline*}
The Adams--Hilton model is therefore given by
\[
  AH_*(M \rtimes S^1) = \bigl(T(a_1,b_1,x_1,y_1,\ldots,a_k,b_k,x_k,y_k,z,w),d\bigr)
\]
where we set $a_i=(-1)^{d_i}s^{-1}e^{d_i}_i$, $b_i=s^{-1}e^{d-d_i}_i$, 
$x_i=(-1)^{d+1}s^{-1}e^{d_i+1}_i$, $y_i=s^{-1}e^{d-d_i+1}_i$, $z=s^{-1}e^{d}$, $w=s^{-1}e^{d+1}$, so that 
$\deg a_i = d_i - 1$, $\deg b_i = d - d_i-1$, $\deg x_i=d_i$, $\deg y_i=d-d_i$, $\deg z = d-1$, $\deg w=d$. Differential~\eqref{dcobar} is given by $d(a_i) = d(b_i) = d(x_i) = d(y_i) = 0$,
\[ 
  d(z) = \sum_{i=1}^k[a_i,b_i]
\]
and
{\small
\begin{align*}
  d(w) &= \sum_{i=1}^k \bigl( (-1)^{d_i} s^{-1}e^{d_i}_i\otimes s^{-1}e^{d-d_i+1}_i
  +(-1)^{d-d_i+1}(-1)^{d_i(d-d_i+1)}s^{-1}e^{d-d_i+1}_i\otimes s^{-1}e^{d_i}_i\\
  +& (-1)^{d_i+1}(-1)^{d-d_i}s^{-1}e^{d_i+1}_i\otimes s^{-1}e^{d-d_i}_i
  +(-1)^{d-d_i}(-1)^{d_i(d-d_i)}s^{-1}e^{d-d_i}_i\otimes s^{-1}e^{d_i+1}_i\bigr)\\
 & = \sum_{i=1}^k\bigl(a_i\otimes y_i+(-1)^{(d_i+1)(d-d_i+1)+d_i}y_i\otimes a_i
  +x_i\otimes b_i+(-1)^{d_i(d-d_i+1)+1}b_i\otimes x_i\bigr)\\
 & =\sum_{i=1}^k\bigl([a_i,y_i]+[x_i,b_i]\bigr).
\end{align*}}%
Therefore any element in $\langle d(z) \rangle$ or $\langle d(w) \rangle$ is trivial in homology since it is a boundary. This induces two relations in $H_{\ast}(\Omega (M \rtimes S^1))$, as claimed.
\end{proof}

\begin{proof}[Proof of Theorem \ref{thm:pgoniffORA}]
(c) $\Rightarrow$ (a).
Suppose that $K=C_p$ or $K=C_p \ast \Delta^q$ for $p \geqslant 4$, $q \geqslant 0$. Since $\Z_K$ is homotopy equivalent to a connected sum of sphere products~\eqref{eq:McGavran},
the implication follows from Proposition~\ref{prop:HopfAlgIso}.

\smallskip

(a) $\Rightarrow$ (c). Suppose that $K$ is a flag complex on $[m]$ such that $H_{\ast}(\Omega \Z_K)$ is a one-relator algebra. If $K^{1}$ is a chordal graph, then $\Z_K$ has the homotopy type of a wedge of spheres \cite[Theorem~4.6]{g-p-t-w16}, and thus $H_{\ast}(\Omega \Z_K)$ is a graded free associative algebra, which is a contradiction.

Therefore assume that $K^{1}$ is not chordal. In particular, there exists a set of vertices $I \subseteq [m]$ such that the full subcomplex $K_I$ is a $p$-cycle, and we enumerate $I=\{b_1,b_2,\ldots,b_p\}$. If $I = [m]$, that is $K = C_p$, then $H_{\ast}(\Omega \Z_K)$ is a one-relator algebra by Proposition~\ref{prop:HopfAlgIso}. 

Assume that $[m] \setminus I \ne \emptyset$. First, we show that each $j \in [m] \setminus I$ is connected to each vertex in $I$. Consider the full subcomplex $K_{I \cup j}$ of $K$, and observe that 
\[
K_{I \cup j} = K_I \cup_{\lk_{I \cup j} (j)} \st_{I \cup j}(j).
\]
Suppose that $K_{I \cup j} \ne K_I \ast j$. Since $K$ is flag, there exists $b_l \in I$ such that there is no edge from $j$ to $b_l$. Form the sequence of adjacent vertices $b_{l+1}, b_{l+2}, \ldots, b_{l+n_1}$, with the convention that $b_{p+1} = b_1$, where $n_1 \geqslant 1$ is the smallest index such that there is an edge from $j$ to $b_{l+n_1}$. Similarly, form sequence of adjacent vertices $b_{l-1}, b_{l-2}, \ldots, b_{l-n_2}$, where $n_2 \geqslant 1$ is again the smallest index such that there is an edge from $j$ to $b_{l-n_2}$. We consider four cases.

(i) Assume that there are no indices $n_1$ and $n_2$ as described above. In this case, there are no edges between $j$ and any vertex in $I$, and so $\lk_{I \cup j}(j) = \emptyset$. Then~\eqref{linkstar} takes the form
\begin{equation*} 
\begin{tikzcd}[column sep=3em,row sep=3em]
  T^{p+1} \ar[r,"i \times \mathrm{id}"] \ar[d,"\pi"'] & \Z_{K_I} \times S^1 \ar[d] \\
  T^p \ar[r] & \Z_{K_{I \cup j}}
\end{tikzcd} \end{equation*}
where the map $i\colon T^p\to \mathcal Z_{K_I}$ is nullhomotopic and therefore $\Z_{K_{I \cup j}} \simeq \Sigma^2 T^p \vee (Z_{K_I} \rtimes S^1)$ by Lemma~\ref{gtlemma}. Since $\mathcal Z_{K_I}$ is homeomorphic to a connected sum of sphere products, Lemma~\ref{lm:HalfSmashGivesRelator} gives that $H_{\ast}(\Omega (\Z_{K_I} \rtimes S^1))$ is not a one-relator algebra, and hence neither is $H_{\ast} \left( \Omega \Z_{K_{I \cup j}} \right)$. 
    
(ii) If $b_{l+n_1} = b_{l-n_2}$, then $\lk_{I \cup j} (j) = b_{l+n_1}$, and $\Z_{K_{I \cup j}} \simeq \Sigma^2 T^{p-1} \vee (Z_{K_I} \rtimes S^1)$ by Lemma~\ref{gtlemma}. Thus $H_{\ast} \left(\Omega \Z_{K_{I \cup j}} \right)$ is not a one-relator algebra.

(iii) When $b_{l+n_1}$ and $b_{l-n_2}$ are adjacent in $K_I$, the link $\lk_{I \cup j}(j) = \{(b_{l+n_1}, b_{l-n_2})\}$, and so $\Z_{K_{I \cup j}} \simeq \Sigma^2 T^{p-2} \vee (Z_{K_I} \rtimes S^1)$ by Lemma~\ref{gtlemma}. Thus $H_{\ast} \left(\Omega \Z_{K_{I \cup j}} \right)$ is not a one-relator algebra.

(iv) Finally, let $b_{l+n_1}$ and $b_{l-n_2}$ be distinct and not adjacent in $K_I$. Then by construction the full subcomplex $K_{\{j, b_{l-n_2}, \ldots, b_{l-1}, b_l, b_{l+1}, \ldots, b_{l+n_1} \}}$ of $K$ is a $(n_1 + n_2 + 2)$-cycle, which is distinct from $K_I$. Therefore by Lemma~\ref{lm:DistinctRetract}, $H_{\ast} \left(\Omega \Z_{K_{I \cup j}} \right)$ is not a one-relator algebra.

In all of the above cases, since the full subcomplex $K_{I \cup j}$ retracts off $K$ and $H_{\ast} \left(\Omega \Z_{K_{I \cup j}} \right)$ is not a one-relator algebra, then neither is $H_{\ast}(\Omega \Z_K)$. This is a contradiction. We therefore conclude that $j$ is connected to each vertex in $K_I$ and therefore $K_{I \cup j} = K_I \ast j$.

Second, we show that if $j_1,j_2 \in [m] \setminus I$, then $j_1$ and $j_2$ are connected by an edge. If not, since both $j_1$ and $j_2$ are connected to each vertex in $I$, then the full subcomplex $K_{\{j_1,b_{i_1},j_2,b_{i_3}\}}$ is a $4$-cycle distinct from the $p$-cycle $K_I$. Therefore, Lemma~\ref{lm:DistinctRetract} implies that since $H_{\ast}\left(\Omega \Z_{I \cup \{j_1,j_2\}} \right)$ is not a one-relator algebra, neither is $H_{\ast}(\Omega \Z_K)$, which is a contradiction.

Therefore any vertex in $[m] \setminus I$ is connected to every vertex in $I$ and to every other vertex in $[m] \setminus I$. Since $K$ is flag, $K = K_I \ast \Delta^q$ for some $q \geqslant 0$.

\smallskip

(c) $\Rightarrow$ (b). Suppose that $K=C_p$ or $K=C_p*\Delta^q$ for $p\geqslant 4$ and $q\geqslant 0$.  Let the $p$-cycle $C_p$ of $K$ be supported on the set of vertices $I = \{b_1,\dots,b_p\}$. By~\eqref{BigradedHochster},
\begin{equation}\label{bh2}
  H_{2-j,2j}(\Z_K) \cong \bigoplus_{J \subseteq [m],\, |J|=j} \widetilde{H}_{1}(K_J).
\end{equation}
Since $K_I$ is a $p$-cycle, we have $\widetilde{H}_{1}(K_I) = \mathbb{Z}$. Because any subcomplex $K_J$ with $J \not= I$ is contractible, $\widetilde{H}_{1}(K_J) = 0$ for $J \not= I$.
It follows that $H_{2-p,2p}(\Z_K) = \mathbb{Z}$ and $H_{2-j,2j}(\Z_K)=0$ for $j\ne p$.

\smallskip

(b) $\Rightarrow$ (c).  Suppose that $H_{2-j,2j}(\Z_K)$ is as described in~(b). Then only one summand on the right hand side of~\eqref{bh2} is $\mathbb{Z}$, and all other summands are zero. The same argument as in the proof of implication (b)~$\Rightarrow$~(c) of Theorem~\ref{thm} shows that $K =C_p*\Delta^q$ for some $q \geqslant 0$.

\smallskip

It remains to prove that if $K=C_p$ or $K =C_p*\Delta^q$ for $q \geqslant 0$ then $H_{-i,2j}(\Z_K)=0$ for $j-i\geqslant 3$. 
Considering bigraded decomposition~\eqref{BigradedHochster},
\[
  H_{-i,2j}(\Z_K) \cong \bigoplus_{J \subseteq [m] ,\, |J|=j} \widetilde{H}_{j-i-1}(K_J).
\]
We claim that all summands on the right hand side with $j-i\geqslant 3$ are equal to~0. Indeed, let $I = \{b_1,\dots,b_p\}$ be the set of vertices of $K$ forming a $p$-cycle, $p \geqslant 4$. Then 
$\widetilde{H}_{j-i-1}(K_I) = 0$ for $j-i\geqslant 3$. Since any full subcomplex 
$K_J$ with $J\ne I$ is contractible, we get $\widetilde{H}_{j-i-1}(K_J) = 0$. Hence, $H_{-i,2j}(\Z_K)=0$ for $j-i\geqslant 3$.
\end{proof}

The following examples illustrate Theorem~\ref{thm:pgoniffORA}.

\begin{example}\label{exampleZK}\
	
\noindent\textbf{1.} Let $K$ be the flag complex in Figure~\ref{fig}~(a). Generator set~\eqref{MACgenerators} for $H_{\ast}(\Omega \Z_K)$ is
\[
    [u_3,u_1], ~[u_4,u_2], ~[u_5,u_4], ~[u_2,[u_5,u_4]].
\]
These satisfy the relations
\[
    [u_3,u_1][u_4,u_2]-[u_4,u_2][u_3,u_1]=0, \quad [u_3,u_1][u_5,u_4]-[u_5,u_4][u_3,u_1]=0
\]
and
\[
    [u_3,u_1][u_2,[u_5,u_4]] - [u_2,[u_5,u_4]][u_3,u_1] = 0
\]
which are derived by using the commutativity relations given in~\eqref{DJ_Kalg}. 
By formula~\eqref{BigradedHochster} we obtain $H_{-2,8}(\Z_K) = \mathbb{Z}^2$ and $H_{-3,10}(\Z_K) = \mathbb{Z}$. Hence, the homological condition of Theorem~\ref{thm:pgoniffORA}~(b) is not satisfied.

\smallskip

\noindent\textbf{2.} Let $K$ be the flag complex in Figure~\ref{fig}~(b). Generator set~\eqref{MACgenerators} for $H_{\ast}(\Omega \Z_K)$ is
\[
    [u_3,u_1], ~[u_4,u_2]
\]
with a single relation $[u_3,u_1][u_4,u_2]-[u_3,u_1][u_4,u_2]=0$. Here $H_{\ast}(\Omega \Z_K)$ is a one-relator algebra, and formula~\eqref{BigradedHochster} gives $H_{-2,8}(\Z_K) = \mathbb{Z}$.

\end{example}

Recall that a simplicial complex $K$ is \emph{Golod} if all cup products and higher Massey products vanish in $H^*(\Z_K)$. A simplicial complex $K$ is \emph{minimally non-Golod} if $K$ is not Golod itself, but for every vertex $\{i\}\in K$ the deletion subcomplex ${K-\{i\}}=K_{[m]\setminus\{i\}}$ is Golod. In the flag case, the properties of $H_*(\Omega\Z_K)$ being a one-relator algebra and $K$ being minimally non-Golod are related as follows.

\begin{proposition}\label{flagmng}
Let $K$ be a flag simplicial complex on~$[m]$. The following conditions are equivalent:
\begin{itemize}
\item[(a)] $\Z_K$ is homotopy equivalent to a connected sum of sphere products, with two spheres in each product;	
\item[(b)] $H_{\ast}(\Omega \Z_K)$ is a one-relator algebra;
\item[(c)] $K$ is minimally non-Golod, or $K = L \ast \Delta^q$ where $L$ is a minimally non-Golod complex and $q \geqslant 0$.
\end{itemize}
\end{proposition}
\begin{proof}
(a) $\Rightarrow$ (b). This follows from Proposition~\ref{prop:HopfAlgIso}.

\smallskip

(b) $\Rightarrow$ (c). This follows from Theorem~\ref{thm:pgoniffORA}, because a $p$-cycle $C_p$ with $p\ge4$ is a minimally non-Golod complex.

\smallskip

(c) $\Rightarrow$ (a). This follows from the fact that if $K$ is minimally non-Golod and flag, then $K=C_p$ with $p\geqslant4$. Indeed, let $K^1$ be one-skeleton of~$K$. If $K^{1}$ is a chordal graph, then $\Z_K$ has the homotopy type of a wedge of spheres~\cite[Theorem~4.6]{g-p-t-w16}, so $K$ is Golod and therefore not minimally non-Golod. Hence, $K$ contains a chordless cycle $C_p$ with $p\geqslant4$. If $\{i\}\in K$ is a vertex not in~$C_p$, then $K-\{i\}$ still contains a chordless cycle and therefore is not Golod. Therefore $K = C_p$, and the result follows from formula~\eqref{eq:McGavran}.
\end{proof}

In the non-flag case the three properties in Proposition~\ref{flagmng} are all different. The implication (a) $\Rightarrow$ (c) holds in the non-flag case by a result of Amelotte~\cite[Theorem~1.2]{amel}, and (a) $\Rightarrow$ (b) is Proposition~\ref{prop:HopfAlgIso}. We illustrate the failure of other implications in the next example.

\begin{example}\sloppy
\begin{figure}[h]
    \centering
    \begin{tikzpicture}
        \coordinate [label=left:$1$] (1) at (0,2);
        \coordinate [label=right:$2$] (2) at (2,2);
        \coordinate [label=right:$3$] (3) at (2,0);
        \coordinate [label=left:$4$] (4) at (0,0);
        \coordinate [label = below: $5$] (5) at (1,1);
        \filldraw[fill=black!20] (2) -- (5) -- (1);
        \filldraw[fill=black!20] (2) -- (5) -- (3);
        \filldraw[fill=black!20] (4) -- (5) -- (3);
        \draw (2) -- (1) -- (4) -- (3) -- (2);
        \foreach \point in {1,2,3,4,5}
            \fill [black] (\point) circle (1.5 pt);
        \coordinate [label = below: $5$] (5) at (1,1);
        \end{tikzpicture}
   \caption{}
    \label{fig2}
\end{figure}
Let $K$ be the simplicial complex in Figure~\ref{fig2}. A calculation similar to Proposition~\ref{prop:HopfAlgIso} using a cellular chain complex for $\Z_K$ shows that $H_{\ast}(\Omega \Z_K)$ is a one-relator algebra given by
\[
  H_*(\Omega\Z_K)\cong\frac{T(a_{13},a_{24},b_{145},b_{1452},b_{1453},b_{14523})}
  {\langle[a_{13},a_{24}]\rangle}.
\]
Moreover, a calculation similar to~\cite[Example~8.4.5]{bu-pa15} shows that the homomorphism $H_*(\Omega\Z_K) \longrightarrow H_*(\Omega DJ_K)$ from~\eqref{haes} maps the generators $a_{13}$, $a_{24}$, $b_{145}$, $b_{1452}$, $b_{1453}$, $b_{14523}$ to the commutators $[u_1,u_3]$, $[u_2,u_4]$, $[u_1,u_4,u_5]$, $[[u_1,u_4,u_5],u_2]$, 
$[[u_1,u_4,u_5],u_3]$, $[[[u_1,u_4,u_5],u_2],u_3]$, respectively, where $[u_1,u_4,u_5]$ is the higher bracket corresponding to the missing face $\{145\}$.

Observe that $K$ is not minimally non-Golod, as $K-\{5\}$ is a $4$-cycle, so implication (b)~$\Rightarrow$~(c) of Proposition~\ref{flagmng} fails in the non-flag case. Furthemore,  (b)~$\Rightarrow$~(a) also fails here, which is seen from the isomorphism of the cohomology ring of~$\Z_K$ with that of $(S^3 \times S^3) \vee S^5 \vee S^6 \vee S^6 \vee S^7$.
\end{example}

The implication (c) $\Rightarrow$ (a) also fails in the non-flag case. Examples of minimally non-Golod complexes $K$ for which $\Z_K$ is not homotopy equivalent to a connected sum of sphere products were constructed by Limonchenko in~\cite[Theorem~2.6]{limo15}.


\begin{thebibliography}{00}
\bibitem{adam56}
Adams, J.\,F. \emph{On the cobar construction}. Proc. Nat. Acad. Sci. USA~42 (1956), 409--412.

\bibitem{ad-hi56}
Adams, J.\,F.; Hilton, P.\,J. \emph{On the chain algebra of a loop space.} Comment. Math. Helv.~30 (1956), 305--330.

\bibitem{amel}
Amelotte, S. \emph{Connected sums of sphere products and minimally non-Golod complexes}. Preprint (2020); arXiv:2006.16320.

\bibitem{be-wu15}
Beben, P.; Wu, J. \emph{The homotopy type of a Poincar\'e duality complex after looping.} Proc. Edinb. Math. Soc. (2)~58 (2015), no.~3, 581--616.

\bibitem{bu-pa15} Buchstaber, V.\,M.; Panov, T.\,E. \textit{Toric Topology.} Mathematical Surveys and Monographs, vol.~204, American Mathematical Society, Providence, RI, 2015.

\bibitem{dy-va73} Dyer,~E.; Vasquez,~A.\,T. \textit{Some small aspherical spaces.} J. Austral. Math. Soc.~16 (1973), 332--352.
	
\bibitem{g-p-t-w16} Grbi\'c~J.; Panov~T.; Theriault~S.; Wu~J. \textit{The homotopy types of moment-angle complexes for flag complexes.} Trans. Amer. Math. Soc. 368 (2016), no.~9, 6663--6682. 

\bibitem{gr-th07}
Grbi\'c~J.; Theriault~S. \emph{The homotopy type of the complement of a coordinate subspace arrangement.} Topology~46 (2007), no.~4, 357--396.

\bibitem{hatc11} 
Hatcher,~A. \textit{Algebraic Topology.} Cambridge University Press, Cambridge, 2002.

\bibitem{ir-ya17}
Iriye,~K.; Yano,~T. \textit{A Golod complex with non-suspension moment-angle complex.} Topology and its Applications~225 (2017), 145--163.
	
\bibitem{limo15} 
Limonchenko~I.\,Yu. \emph{Families of minimally non-Golod complexes and polyhedral products.} Dal'nevost. Mat. Zh.~15 (2015), no.~2, 222--237  (in Russian); arXiv:1509.04302.

\bibitem{lynd50} Lyndon~R.\,C. \textit{Cohomology theory of groups with a single defining relation.} Ann. of Math.~(2)~52 (1950), 650--665.

\bibitem{mcga79}
McGavran, D. \emph{Adjacent connected sums and torus actions.}
Trans. Amer. Math. Soc.~{251} (1979), 235--254.

\bibitem{pa-ra04}
Panov~T.\,E;  Ray~N. \textit{Categorical aspects of toric topology.} Toric topology, Contemp. Math. 460, Amer. Math. Soc., Providence, RI, 2008, pp. 293--322.

\bibitem{pa-ve16} Panov~T.\,E.; Veryovkin Ya.\,A. \textit{Polyhedral products and commutator subgroups of right-angled Artin and Coxeter groups.} Sb. Math.~207 (2016), no.~11, 1582--1600.

\bibitem{vere16} 
Veryovkin Ya.\,A. \emph{Pontryagin algebras of some
moment-angle-complexes}.  Dal'nevost. Mat. Zh. (2016), no.~1,
9--23 (in Russian); arXiv:1512.00283.
\end{thebibliography}
\end{document}